\documentclass[leqno]{amsart}
\usepackage{amsmath,amsfonts,amsthm,amssymb,indentfirst,epic,url,centernot}

\newtheorem{theorem}{Theorem}
\newtheorem{lemma}[theorem]{Lemma}
\newtheorem{corollary}[theorem]{Corollary}
\newtheorem{proposition}[theorem]{Proposition}
\newtheorem{definition}[theorem]{Definition}
\newtheorem{question}[theorem]{Question}

\newcommand{\Q}{\mathbb Q}

\renewcommand{\r}{\mathrm}
\DeclareMathOperator{\lt}{length}
\DeclareMathOperator{\soc}{socle}
\DeclareMathOperator{\card}{card}

\begin{document}

\begin{center}
\texttt{Comments, corrections, and related references welcomed,
as always!}\\[.5em]
{\TeX}ed \today
\vspace{2em}
\end{center}

\title{Minimal faithful modules over Artinian rings}
\thanks{After publication of this note, updates, errata, related
references etc., if found, will be recorded at
\url{http://math.berkeley.edu/~gbergman/papers/}\,.
}


\author{George M. Bergman}
\address{University of California\\
Berkeley, CA 94720-3840, USA}
\email{gbergman@math.berkeley.edu}

\begin{abstract}
Let $R$ be a left
Artinian ring, and $M$ a faithful left $\!R\!$-module which
is minimal, in the sense that no proper submodule or
proper homomorphic image of $M$ is faithful.

If $R$ is local, and $\soc(R)$ is central in $R,$
we show that $\lt(M/J(R)M) + \lt(\soc(M))\leq \lt(\soc(R))+1,$
strengthening a result of T.\,Gulliksen.

We then consider a ring $R,$ still left Artinian, but
not necessarily local, and not necessarily having central socle.
If $R$ is a finite-dimensional algebra over an
algebraically closed field, we get an inequality similar to the above,
with the length of $\soc(R)$ interpreted as its length as a bimodule,
and the final summand $+1$ replaced
by the Euler characteristic of a bipartite graph determined
by the bimodule structure of $\soc(R).$
That inequality holds, more generally, if, rather than assuming
$k$ algebraically closed, we assume that
$R/J(R)$ is a direct product of full matrix algebras over $k,$
and exclude the case where $k$ has small finite cardinality.
Examples show that the restriction on the cardinality
of $k$ is needed; we do not know whether the other hypotheses
can be significantly weakened.

We end with a section, essentially independent of what
precedes, on faithful modules with only one of these
minimality properties,
i.e., having no faithful proper submodules {\em or} having
no faithful proper homomorphic images.
Here the conclusion is more straightforward:
The length of $M/J(R)M$ in the former case, and of
$\soc(M)$ in the latter, is $\leq\lt(\soc(R))$ (where this
again means length as a bimodule).
We also show that every faithful module over a left Artinian ring
has a faithful submodule with the former minimality condition,
and a faithful factor module with the latter; the proofs are
based on some general results on decompositions of modules.
\end{abstract}

\maketitle

\section{Background and motivation}\label{S.background+}
This paper arose as a tangent to the unpublished
note~\cite{cmtg_mxs}, which examines the
question of which commutative Artinian rings $R$ have the
property that every faithful $\!R\!$-module $M$ has
length greater than or equal to that of $R.$
(If $R$ is a commutative algebra over a field $k,$
it is known that this is true if $R$ can be generated
over $k$ by $2$ elements, but
false for $\!4\!$-generator algebras;
it is an open problem whether it holds for $\!3\!$-generator algebras.
For more on this,
see~\cite{cmtg_mxs} and~\cite[Chapter~5]{O+C+V}.)

In studying that question, it is
natural to focus on faithful modules $M$
no proper factor-modules or submodules of which are faithful.
I obtained some results showing that such $M$ must have
small ``top'' $M/J(R)M$ and ``bottom'' $\soc(M);$
Luchezar Avramov
then pointed me to a 1972 paper of Tor Gulliksen, \cite{TG}, which
obtained a stronger result; I found, in turn, that
Gulliksen's bound could be strengthened,
and that the strengthened result could be applied to
a wider class of rings (which, in particular, need not be commutative).
This will be done in \S\ref{S.cm} below.

Some details:
the relevant result of Gulliksen's paper is that if $R$ is a
commutative local Artinian ring, and $M$ a faithful $\!R\!$-module
no proper submodule or homomorphic image of which is faithful, then
each of the semisimple $\!R\!$-modules
$M/J(R)M$ and $\soc(M)$ has length less than or equal to
that of $\soc(R),$ with at least one of these inequalities
strict unless $M\cong R.$
His proof is, in effect,
a lemma in linear algebra, about bilinear maps
$A\times B\to C$ of finite-dimensional vector spaces over a field,
which have the property that every nonzero element of $A$ acts
nontrivially, but which lose this property both on
restriction to any proper subspace
of $B,$ and on composition with any noninjective map out of $C;$
though he only states it
for the natural map $\soc(R)\times M/J(R)M\to\soc(M)$
of vector spaces over the field $R/J(R).$
In \S\ref{S.cm} we show that,
in the general linear algebra setting, one has the stronger inequality
$\dim(A)\geq\dim(B)+\dim(C)-1.$
We then, like Gulliksen, apply this result
to maps $\soc(R)\times M/J(R)M\to\soc(M);$ here, rather than assuming
the local ring $R$ commutative, it is only
necessary to assume $\soc(R)$ central in $R.$

Subsequent sections obtain inequalities of a similar nature for
modules over not necessarily local Artinian rings $R,$ as sketched
in the Abstract.

\section{Faithful modules over local Artin rings with central socles}\label{S.cm}

Before formulating the promised linear algebra result, let us note
that for finite-dimensional vector spaces $B$ and $C$ over a
field $k,$
to give a linear map $a:B\to C$ is equivalent to giving an
element of $B^*\otimes_k C.$
Hence a $\!k\!$-vector space $A,$ given with a $\!k\!$-bilinear
map $A\times B\to C$
such that every nonzero element of $A$ induces a nonzero map
$B\to C$ is equivalent to a subspace $A\subseteq B^*\otimes_k C.$
This is a more symmetric situation; so we shall formulate the
linear algebra result in that form, with $B^*$ re-named $B.$

As noted in the preceding section, the bilinear maps $A\times B\to C$
of interest are those such that restriction to any proper subspace
of $B,$ or composition with the natural map into any proper
homomorphic image of $C,$ kills the action of some element of $A.$
With $B$
dualized as above, this becomes the condition that
passing to a homomorphic image of either $B$ or $C$ has that effect.
Now a minimal proper homomorphic image of $B$ has the form $B/k b$
for some $b\in B-\{0\},$ so the condition that passage to any
such homomorphic image
kills some element of $A$ says that for each nonzero $b\in B,$ the
space $A$ contains a nonzero element of the form $b\otimes c.$
Likewise, the condition that passing to any proper homomorphic
image $C/k c$ of $C$ kills some element of $A$ means that for each
nonzero $c\in C,$ the space $A$ contains a nonzero element
of the form $b\otimes c.$
This leads to the formulation of the next result.

My original proof required that the field $k$ have cardinality
at least $\max(\dim_k(B),\,\dim_k(C)).$
For the present proof, I am indebted to
Cl\'{e}ment de Seguins Pazzis \cite{overflow}.

\begin{proposition}\label{P.cm_bilin}
Let $k$ be a field, and $B$ and $C$ nonzero finite-dimensional vector
spaces over $k.$
Suppose $A\subseteq B\otimes_k C$ is a subspace such that
\begin{equation}\begin{minipage}[c]{35pc}\label{d.cm_b}
$(\forall\,b\in B-\{0\})\,(\exists\,c\in C-\{0\})\ b\otimes c\in A,$
\end{minipage}\end{equation}
and
\begin{equation}\begin{minipage}[c]{35pc}\label{d.cm_c}
$(\forall\,c\in C-\{0\})\,(\exists\,b\in B-\{0\})\ b\otimes c\in A.$
\end{minipage}\end{equation}
Then
\begin{equation}\begin{minipage}[c]{35pc}\label{d.cm_dim}
$\dim_k(A)\ \geq\ \dim_k(B)+\dim_k(C)-1.$
\end{minipage}\end{equation}
\end{proposition}

\begin{proof}
(After de Seguins Pazzis \cite{overflow}.)

If $\dim_k(B)$ or $\dim_k(C)$ is $1,$ then~\eqref{d.cm_c},
respectively~\eqref{d.cm_b}, says that $A=B\otimes C,$ and
the desired result is immediate.
So let $m,$ $n>1,$ assume inductively that the result is
known when $B$ has dimension $m-1$ and $C$ has dimension $n-1,$
and suppose we are in a situation with $\dim_k(B)=m$ and $\dim_k(C)=n.$
We consider two cases:

Case~1.
There exists a linear functional $f:B\to k$ such that the induced
map $f\otimes\,\r{id}_C: B\otimes C\to k\otimes C\cong C$
carries $A\subseteq B\otimes C$ surjectively onto $C.$

Then taking any basis $\{c_1,\dots,c_n\}$ of $C,$ we can find
elements $a_1,\dots,a_n\in A$ whose images under
$f\otimes\,\r{id}_C$ are $c_1,\dots,c_n.$
On the other hand, letting $b_1,\dots,b_{m-1}$ be any basis of
$\r{ker}(f)\subseteq B,$ we can find, by~\eqref{d.cm_b},
nonzero $c'_1,\dots,c'_{m-1}\in C$ such
that $b_1\otimes c'_1,\dots,b_{m-1}\otimes c'_{m-1}\in A.$
We claim that
$a_1,\dots,a_n,\,b_1\otimes c'_1,\dots,b_{m-1}\otimes c'_{m-1}\in A$
are linearly independent.
Indeed, given any linear dependence
relation among these elements, if we apply
$f\otimes\,\r{id}_C$ to it, this will annihilate the last $m-1$ terms,
and looking at the remaining terms, we conclude from the linear
independence of $c_1,\dots,c_n$ that the coefficients
of $a_1,\dots,a_n$ in the relation must be zero.
Given this fact, if for any $i\leq m-1$
we let $f_i: B\to k$ be a linear functional which takes $b_i$
to $1$ and all other $b_j$ to $0,$ then application of
$f_i\otimes\,\r{id}_C$ to our relation shows that the coefficient
of $b_i\otimes c'_i$ is also zero.
So the indicated $m+n-1$ elements of $A$
are linearly independent, establishing~\eqref{d.cm_dim} in this case.

The negation of the condition of Case~1 says that for {\em every}
linear functional $f:B\to k,$ a certain conclusion holds.
But to finish the proof, it will be enough to assume this for
{\em some nonzero} $f,$ as we do in

Case~2.
There exists a nonzero linear functional $f:B\to k$ such that the
induced map $f\otimes\,\r{id}_C: B\otimes C\to k\otimes C\cong C$
carries $A\subseteq B\otimes C$ into a proper subspace $C_0\subseteq C.$

Without loss of generality, we can take $C_0$ to have dimension $n-1.$
Let $B_0=\r{ker}(f),$ which has dimension $m-1.$
Then our assumption on $f$ implies
\begin{equation}\begin{minipage}[c]{35pc}\label{d.either_or}
For any element of $A$ of the form $b\otimes c,$
either $b\in B_0$ or $c\in C_0.$
\end{minipage}\end{equation}

Now let $p_B$ be any retraction $B\to B_0,$
let $p_C$ be any retraction $C\to C_0,$ and let $A_0$ be
the image of $A$ under $p_B\otimes p_C: B\otimes C\to B_0\otimes C_0.$
(It is not asserted that $A_0\subseteq A.)$
I claim that the analogs of~\eqref{d.cm_b} and~\eqref{d.cm_c} hold
with $A_0,$ $B_0,$ $C_0$ in the roles of $A,$ $B$ and $C.$
Indeed, given nonzero $b\in B_0,$ let $b'\in B$ be chosen which
projects to $b$ under $p_B,$ but is {\em not} in $B_0;$
and use~\eqref{d.cm_b} to find a nonzero $c\in C$ such
that $b'\otimes c\in A.$
Then by~\eqref{d.either_or}, $c\in C_0,$ so
$c$ is fixed by $p_C;$ so applying
$p_B\otimes p_C$ to $b'\otimes c,$ we get $b\otimes c\in A_0;$
the analog of~\eqref{d.cm_b}.
The symmetric argument gives the analog of~\eqref{d.cm_c}.

Hence by our inductive assumption, $\dim(A_0)\geq (m-1)+(n-1)-1,$
which is short by $2$ of the desired lower bound on
the dimension of $A.$
Since $A_0$ is the image of $A$ under $p_B\otimes p_C,$
it will suffice to find two linearly independent elements of
$A$ which are in the kernel of that map.
To do so, let $b$ span $\r{ker}(p_B)$ and $c$ span $\r{ker}(p_C),$
and use~\eqref{d.cm_b} and~\eqref{d.cm_c} to
find nonzero elements $b\otimes c'$ and $b'\otimes c$ of $A.$
From~\eqref{d.either_or}, one sees that these are linearly independent,
completing the proof.
\end{proof}

Dualizing $B,$ to pass from this result to a statement about
linear maps, as discussed earlier, we get

\begin{corollary}\label{C.cm}
Suppose $A,$ $B$ and $C$ are finite-dimensional vector spaces
over a field $k,$ and $f:A\times B\to C$ a bilinear map, such
that every nonzero element of $A$ induces a nonzero map $B\to C,$
but such that this property is lost on restriction to any
proper subspace of $B,$ and likewise on composition with the map to
any proper homomorphic image of $C.$
Then $\dim(A)\geq\dim(B)+\dim(C)-1.$\qed
\end{corollary}

We shall now deduce the asserted generalization of Gulliksen's result.
(Incidentally, Gulliksen does not explicitly say in \cite{TG} that
his rings are commutative; but this is apparent
from the techniques he uses, e.g., the duality called on
at the top of \cite[p.\,79]{TG}; and is also evidenced
by the fact that commutativity
is one of the properties he verifies for the matrix
example of \cite[Theorem~2]{TG}.
In the present note,
rings are not assumed commutative unless this is stated.)

Recall that the {\em socle} of a left or right module $M$
is the sum of its simple submodules.
By the left and right socles of a ring $R,$ we
understand $\soc({_R}R)$ and $\soc(R_R),$ the socles of $R$
as a left and as a right module, each of which is a $\!2\!$-sided
ideal of $R.$
On the other hand, by {\em the socle} of $R,$ $\soc(R),$ we shall mean
the sum of all its minimal nonzero $\!2\!$-sided ideals,
which, for $R$ left or right Artinian, is the intersection of its right
and left socles.
(So, for example, in the algebra of $n\times n$ upper
triangular matrices over a field,
the left socle is the ideal of matrices with support in
the top row, the right socle is the ideal of matrices with
support in the last column, and $\soc(R)$ is the ideal
of matrices with support in the upper right-hand corner.)

We denote the Jacobson radical of a ring $R$ by $J(R).$

\begin{theorem}[cf.\ Gulliksen {\cite[Lemma~2]{TG}}]\label{T.TG}
Suppose $R$ is a left Artinian local ring such that
$\soc(R)$ is central in $R,$ and let $M$ be any
faithful left $\!R\!$-module such that no proper submodule or
homomorphic image of $M$ is faithful.
Then
\begin{equation}\begin{minipage}[c]{35pc}\label{d.cm_local_dim}
$\dim_{R/J(R)}(M/J(R)M)\ +\ \dim_{R/J(R)}(\soc(M))\ \leq
\ \dim_{R/J(R)}(\soc(R))\ +\ 1.$
\end{minipage}\end{equation}
\end{theorem}

\begin{proof}
$J(R)$ annihilates $M/J(R)M$ and $\soc(M)$
as left $\!R\!$-modules, and annihilates the $\!(R,R)\!$-bimodule
$\soc(R)$ on both sides; so the first two
become left vector spaces, and the latter a bimodule,
over the division ring $R/J(R).$
The statement that $\soc(R)$ is central in $R$ says that
$R/J(R)$ has the same action on the two sides of $\soc(R),$ from
which it immediately follows that the division
ring $R/J(R)$ must be a field $k.$
Writing $A=\soc(R),$ $B=M/J(R)M,$ $C=\soc(M),$
the left module operation of $R$ on $M$
induces an action by which $A$ carries $B$ into $C,$ giving
a $\!k\!$-bilinear map $A\times B\to C.$
Because $M$ is faithful, every nonzero element of $A$ gives
a nonzero map $B\to C,$ while
the minimality assumptions on $M$ imply that this property of our
bilinear map is lost when one passes to any proper subspace of $B$
or proper homomorphic image of $C.$
(Indeed, if $B_0$ is a proper $\!k\!$-subspace of $B,$
equivalently, a proper $\!R\!$-submodule, then its
inverse image in $M$ is a proper submodule, hence non-faithful.
The kernel of the action of $R$ on this submodule must meet $\soc(R)=A,$
and any $a\in A$ in that kernel acts trivially on $B_0.$
The dual argument applies to proper homomorphic images $C_0$ of $C.)$
That $B$ and $C$ are finite-dimensional is also not hard to see
from these minimality conditions.
(Some general results of which this finite dimensionality is a special
case are developed in \S\ref{S.sub_or_hom}.)

Hence Corollary~\ref{C.cm} applies, and gives the desired inequality.
\end{proof}

{\em Remarks}:
If $R$ is a left Artinian local ring, then
a {\em sufficient} condition for $\soc(R)$ to be central is
that $R$ be an algebra over a field $k,$ and that the residue field
of $R$ be $k$ itself.
Such an $R$ can be far from commutative;
for instance, over any field we can take for $R$ the ring of upper
triangular $n\times n$ matrices over $k$ with scalar main diagonal.
On the other hand, if $R$ contains a field $k$ which
is not central in $R,$ but which again
maps isomorphically onto $R/J(R),$
then $\soc(R)$ may or may not be central.
For example, if we take a twisted polynomial ring $k[x;\theta]$
where $\theta$ is an automorphism of $k$ of finite order $d>1,$
and look at its local factor-ring $k[x;\theta]/(x^{n+1})$ for some
$n>0,$ then $k$ is not central in $R,$ but the socle, $k\,x^n,$
is central if and only $n$ is a multiple of $d.$
Likewise, given a field $k$ and automorphisms $\theta_2,\dots,\theta_n,$
we can generalize the triangular-matrices example
by letting $R$ be the ring of upper-triangular $n\times n$ matrices
$((a_{ij}))$ over $k$
satisfying $a_{ii}=\theta_i(a_{11})$ for $2\leq i\leq n.$
This will be Artinian and local with residue field isomorphic
to $k,$ and its socle, $k\,e_{1n},$ will be
central if and only if $\theta_n=\r{id}.$

Turning back to Proposition~\ref{P.cm_bilin}, one may ask what a
minimal subspace $A\subseteq B\otimes C$ satisfying~\eqref{d.cm_b}
and~\eqref{d.cm_c} can look like.
Easy examples are the spaces of the form $B\otimes c + b\otimes C$
for arbitrary nonzero elements $b\in B,$ $c\in C.$
These spaces have dimension exactly $\dim(B)+\dim(C)-1,$
since $B\otimes c$ and $b\otimes C$
intersect in the one-dimensional subspace spanned by $b\otimes c.$
In matrix notation, this example can be pictured as the vector space
of $m\times n$ matrices with support in the union of the first
row and the first column.

Using such
matrix notation, one can describe further minimal families.
For every positive integer $q\leq\min(m,n),$ I claim that the space
$A_q$ of $m\times n$ matrices with support in the union of the first
$q$ rows and the first $q$ columns, such
that the first $q$ entries along the main diagonal are all equal,
satisfies~\eqref{d.cm_b} and~\eqref{d.cm_c}.
To see ~\eqref{d.cm_b}, note that given any nonzero column vector
$b=(\beta_1,\dots,\beta_m)^T,$ we can find a nonzero
row vector $c=(\gamma_1,\dots,\gamma_q,\,0,\dots,0)$ such that
$\beta_1\gamma_1=\dots=\beta_q\gamma_q.$
Indeed, if at least one of $\beta_1,\dots,\beta_q$ is zero, we can
choose the $\!\gamma\!$'s not all zero so that all
products $\beta_i\gamma_i$ are
zero, while if all of $\beta_1,\dots,\beta_q$ are nonzero, we can
choose the $\!\gamma\!$'s so that all $\beta_i\gamma_i$ equal~$1.$
In either case, the $m\times n$ matrix $b\otimes c$ will be
nonzero and lie in $A_q.$
Similarly, given a nonzero row vector $c,$ we can find a nonzero
column vector $b$ such that $b\otimes c\in A_q,$
proving~\eqref{d.cm_c}.

It is not hard to see that for every column
vector $b\in k^m$ at most one of whose first $q$ entries is zero,
and which has at least one nonzero entry after the first $q$
entries, the above construction gives (up to scalars) the
only $c$ such that $b\otimes c\in A_q.$
(The condition that at least one entry of $b$ after the
first $q$ be nonzero guarantees that every $c$ with
$b\otimes c\in A_q$ must live in its first $q$ entries, which
we need to get uniqueness.)
Likewise, if we are given $c$ at most one of whose
first $q$ entries is zero, and having at least
one later nonzero entry, there is up to scalars
only one $b$ such that $b\otimes c\in A_q.$
Hence, any subspace of $A_q$ that satisfies~\eqref{d.cm_b}
and~\eqref{d.cm_c} must contain all the elements $b\otimes c$
of the two sorts just described.
Combining these observations, one can deduce that if at least one
of $m,$ $n$ is $>q,$ and $k$ has $>2$ elements, then $A_q$ indeed
has no proper subspaces satisfying~\eqref{d.cm_b} and~\eqref{d.cm_c}.
(The condition that $k$ has $>2$ elements is used to get every
row or column as a sum of rows or columns satisfying appropriate
conditions on which entries are nonzero.
For $k$ the $\!2\!$-element field and $q>2,$ there do
in fact exist proper subspaces of $A_q$
satisfying~\eqref{d.cm_b} and~\eqref{d.cm_c}.)
In the remaining case, $m=n=q,$ one finds that the same argument works
without the proviso in the first sentence of this paragraph
about ``at least one nonzero entry after the first $q$ entries''.

So if $k$ has more than $2$ elements,
these constructions do give minimal examples.
I don't know whether, conversely, every minimal subspace of a
tensor product $B\otimes_k C$
satisfying~\eqref{d.cm_b} and~\eqref{d.cm_c} is, with respect
to some bases of $B$ and $C,$ of one of these forms.

It is curious that the bound $\dim(A)\geq\dim(B)+\dim(C)-1$
of Proposition~\ref{P.cm_bilin}
also appears in \cite[Proposition~1.3, case $k=1$]{DE},
for subspaces $A\subseteq B\otimes C$
subject to a different condition involving rank-$\!1\!$ entities;
namely, that $A$ not be contained in the kernel of any tensor product
$f\otimes g$ of linear functionals $f\in A^*,$ $g\in B^*$
\cite[Proposition-Definition~1.1(2), case $k=1$]{DE}.
But I cannot see a relation between the two results.
Indeed, the result in \cite{DE} is proved only over an algebraically
closed field, and it is noted that it fails without
that condition; while the last clause of \cite[Proposition~1.3]{DE}
implies that every minimal submodule $A$ of the sort considered there
has dimension exactly $\dim(B)+\dim(C)-1,$ though for
our condition, we have seen the opposite.
Finally, most of the minimal examples noted above for our condition
do not satisfy the condition of \cite{DE}, precisely because
their matrix representations involve $\!0\!$'s in fixed locations.
So if there is a relation between the two
results, it must be a subtle one.
\vspace{.3em}

Turning back to
Theorem~\ref{T.TG}, and restricting attention to commutative $R,$
observe that if we wish to extend that result to a
commutative not-necessarily-local Artinian ring $R,$ then
this will be a direct product of commutative local
rings $e_1 R\times\dots\times e_d R,$ where the $e_i$ are the
minimal idempotents of $R.$
By summing the inequalities~\eqref{d.cm_local_dim} for these $d$ rings,
we get a similar inequality, with the final $+1$ replaced by $+d.$
The dimensions in the inequalities we sum are with respect to different
base fields $e_i R/J(e_i R);$ the most natural way to refer
to these dimensions is as the {\em lengths} of those modules.
So in this situation, the analog of~\eqref{d.cm_local_dim} is
\begin{equation}\begin{minipage}[c]{35pc}\label{d.cm_artin_dim}
$\lt(M/J(R)M)\ +\ \lt(\soc(M))\ \leq\ \lt(\soc(R))\ +\ d.$
\end{minipage}\end{equation}
This points toward the form of the inequalities we will obtain in
the next two sections, for not-necessarily-commutative
Artinian $R.$

\section{The general Artinian case -- preliminary steps}\label{S.lt_art,pre}

The result we are now aiming for
will again be an application of a statement about
a bilinear map, but with semisimple left modules in the roles
of $B$ and $C,$ and a bimodule in the role of $A.$
In this case, I don't see how to turn statements
about bilinear maps into statements about subobjects of
tensor products, so we shall develop directly the
``$A\times B\to C$'' result analogous to Corollary~\ref{C.cm}.
(However, there will be an obvious symmetry between
what we do with $B$ and with $C;$ which
suggests that I am missing some way that these can be unified.)

Throughout this and the next section, we will therefore assume that
\begin{equation}\begin{minipage}[c]{35pc}\label{d.setup}
$S$ and $T$ are semisimple Artinian rings, ${_S}B$ and ${_T}C$
are left modules of finite lengths, $_T A_S$ is a
bimodule,
and $h:{_T}A_S \times {_S}B\to {_T}C$ a balanced bilinear map, such that
\end{minipage}\end{equation}
\begin{equation}\begin{minipage}[c]{35pc}\label{d.nondeg}
Every nonzero $a\in A$ induces a nonzero map of
abelian groups $h(a,-):B\to C.$
\end{minipage}\end{equation}

In our final application, $S$ and $T$ will be the same, namely
$R/J(R);$ but keeping them distinct until then
will make our manipulations clearer.

For $a\in A,$ we shall think of $h(a,-): B\to C$ as
``the action of $a$'', and thus
speak of elements of $A$ as annihilating certain elements of $B,$
having certain elements of $C$ in their images, etc.;
thus, we will seldom mention $h$ explicitly.

When we speak of the kernel or image of a
{\em family} of maps, we shall mean
the intersection of the kernels of those maps, respectively, the
sum of their images.

We now state the conditions corresponding
to~\eqref{d.cm_b} and~\eqref{d.cm_c}.
\begin{equation}\begin{minipage}[c]{35pc}\label{d.b}
For every maximal submodule $B_0\subseteq B,$ there exists a nonzero
$a\in A$ which annihilates $B_0;$ equivalently, such that the
kernel of $a\,S$ is precisely $B_0.$
\end{minipage}\end{equation}
\begin{equation}\begin{minipage}[c]{35pc}\label{d.c}
For every simple submodule $C_0\subseteq C,$ there exists a nonzero
$a\in A$ which carries $B$ into $C_0;$ equivalently, such that
the image of $Ta$ is precisely $C_0.$
\end{minipage}\end{equation}

The statements of equivalence follow by combining the fact that
$\r{ker}(a\,S)$ and $\r{im}(T\,a)$ are submodules
of ${_S}B$ and ${_T}C$ with the assumptions that $B_0$ is
maximal and $C_0$ simple.

Let us now see what happens to conditions~\eqref{d.b} and~\eqref{d.c}
when we write $S$ and $T$ as direct products of simple Artin rings,
and decompose $A,$ $B$ and $C$ accordingly.
Say the decompositions of the identity elements
of $S$ and $T$ into minimal central idempotents
are $1_S=e_1+\dots+e_m$ and $1_T=f_1+\dots+f_n,$
so that $S\cong\prod_i e_i S$ and $T\cong\prod_j f_j T$ as rings.
Then in the situation of~\eqref{d.c}, the simple submodule
$C_0$ is necessarily contained in some summand $f_j\,C,$
so the $a$ that we get must satisfy $a=f_j\,a.$
Moreover, if we take some $e_i$ such that $a\,e_i\neq 0,$
then the image of $a\,e_i$ must also generate the simple
submodule $C_0;$ so replacing $a$ by $a\,e_i$
if necessary, we can assume $a\in f_j\,A\,e_i.$

Let us now fix some $f_j.$
Then we can deduce from~\eqref{d.c} and the above observations that
every simple submodule of $f_j\,C$ is generated by the image of an
element of $f_j\,A\,e_i$ for some $i\leq m.$
It would be nice if we could reverse the order of
quantifications, and say that there exists some
$i\leq m,$ such that every simple submodule of $f_j\,C$ is generated by
the image of an element of $f_j\,A\,e_i;$
but that is too much to hope for.
However, it turns out that, under some weak assumptions,
we can show that there exists an $i$ such that {\em enough}
of the simple submodules of $f_j\,C$ are images of
elements of $f_j\,A\,e_i$ to suit our purposes.
The key concept is introduced in the definition below.
In that definition, think of $X$ as either $B$ or $e_i\,B,$
of $Y$ as $f_j\,C,$ of $W$ as $f_j\,A,$ or $f_j\,A\,e_i,$
and of $V$ as $f_j\,T\,f_j.$
The adjective ``left'' in that definition, and ``right'' in the one
that follows it, refer to whether we are thinking of fixing
the left factor $f_j$ in the product-symbol $f_j\,A\,e_i$
(as in the above discussion), or the right factor $e_i.$

\begin{definition}\label{D.lt_strong}
If $W$ is a set of homomorphisms from an abelian
group $X$ to a left module $Y$ over a simple Artin ring $V,$
and $N$ is a positive real number, then we shall call $W$
{\em left $\!N\!$-strong} if for every family of $\leq N$
proper submodules of $Y,$ there exists $w\in W$ such that
$V w(X)$ is a simple submodule of $Y$ not
contained in any member of that family.
We shall call $W$ {\em left strong} if it is left $\!N\!$-strong
for all positive real numbers $N.$
We shall at times use ``left $\!\infty\!$-strong'' as a
synonym for ``left strong''.
\end{definition}

(We have taken $N,$ when finite, to be a real number rather than an
integer so that some later statements can be formulated more
conveniently; e.g., so that we can say ``$\!N/d\!$-strong''
rather than ``$\!\lfloor N/d \rfloor\!$-strong''.)

We similarly define a variant of condition~\eqref{d.b}, namely

\begin{definition}\label{D.rt_strong}
If $W$ is a set of homomorphisms from a left module
$X$ over a simple Artin ring $U$ to an abelian group $Y,$
and $N$ a positive real number, then we shall call $W$
{\em right $\!N\!$-strong} if for every family of $\leq N$
nonzero submodules of $X,$ there exists $w\in W$ such that
$\r{ker}(w\,U)$ is a maximal submodule of $X,$
and does not contain any member of that family.
We shall call $W$ {\em right strong} if it is right $\!N\!$-strong
for all positive real numbers $N.$
We shall at times use ``right $\!\infty\!$-strong'' as a
synonym for ``right strong''.
\end{definition}

The observations following~\eqref{d.b} and~\eqref{d.c} will
yield statements to the effect that for appropriate $N,$
the sets $\bigcup_i\,f_j\,A\,e_i$
of maps $B\to f_j\,C$ are left $\!N\!$-strong,
and that the sets $\bigcup_j\,f_j\,A\,e_i$ of maps $e_i\,B\to C$
are right $\!N\!$-strong.
(We postpone the details for the moment.)
The next lemma shows the virtue of the $\!N\!$-strong condition:
it is inherited, in slightly weakened form, by at least one
term of any such union.

\begin{lemma}\label{L.strong_union}
Suppose, in the context of Definition~\ref{D.lt_strong},
that $W$ is the union $W_1\cup\dots\cup W_d$ of $d$ subsets.
Then if $W$ is left $\!N\!$-strong, for $N$ a positive
real number or $\infty,$ then
at least one of the $W_i$ is left $\!N/d\!$-strong.
\textup{(}So for $N=\infty,$ this says that
if $W$ is left strong, so is one of the $W_i.)$

Likewise, if in the context of Definition~\ref{D.rt_strong},
$W=W_1\cup\dots\cup W_d$ is right $\!N\!$-strong, then
at least one of the $W_i$ is right $\!N/d\!$-strong.
\end{lemma}

\begin{proof}
Let us prove the statements for finite $N$ in contrapositive form.
In the case of the first assertion, if left $\!N/d\!$-strength
fails for each of the $W_i,$ then for each $i$ we can find
a family of $\leq N/d$ proper submodules of $Y$ whose union
contains all simple submodules generated by images of members of $W_i.$
Taking the union over $i$ of these sets of submodules, we get
a set of $\leq N$ proper submodules of $Y$ whose union
contains all such simple submodules, showing that
$W$ is not left $\!N\!$-strong.
The dual argument works for the second statement.
The cases with $N=\infty$ follow from the cases for finite $N.$
\end{proof}

The next lemma gives the postponed argument which will allow us
to obtain from conditions~\eqref{d.b} and~\eqref{d.c} statements
about $\!N\!$-strength; i.e., which will show that ``{\em all} simple
submodules'' as in~\eqref{d.c} entails ``a family of submodules not
contained in the union of $N$ proper submodules'', for appropriate $N.$

We have spoken above of the {\em minimal central} idempotents
of a semisimple ring $T;$ we shall now deal with
{\em minimal} idempotents (where an idempotent $e$ of a
ring $R$ is considered
less than or equal to an idempotent $f$ if $eRe\subseteq fRf;$
equivalently, if $e = ef = fe.$
By ``minimal'' we of course mean ``minimal nonzero''.)
Note that if $f$ is a minimal central idempotent of a semisimple
artinian ring $T,$ then
$V=f\,Tf$ is a simple Artinian ring, that is,
a matrix ring $\r{Matr}_{n,n}(D)$ over some division ring $D$
\cite[Theorem~3.5]{TYL1};
and that by taking any minimal idempotent $f'\in \r{Matr}_{n,n}(D),$
one can recover $D$ up to isomorphism as
$f'\,\r{Matr}_{n,n}(D)\,f'= f'(f\,Tf)f'=f'Tf'.$

\begin{lemma}\label{L.not_union}
Let $V$ be a simple Artin ring, $Y$ a left $\!V\!$-module,
$f'$ a minimal idempotent of $V,$ and $N=\card(f'\,Vf').$

Then $Y$ cannot be written as a finite union
$Y_1\cup\dots\cup Y_n$ of proper submodules with $n\leq N.$

Likewise, there is no finite family of nonzero submodules
$Y_1,\dots,Y_n$ of $Y$ with $n\leq N$
such that every maximal submodule contains some $Y_i.$
\end{lemma}

\begin{proof}
The first of the above assertions is equivalent to saying that
there does not exist a finite
family of $\leq N$ proper submodules of $Y$ such
that every simple submodule of $Y$ is contained in a member of
that family;
which is a statement about the lattice of submodules of $Y.$
Now letting $D=f'\,Vf',$ a division ring, we know that $V$ is
Morita equivalent to $D,$ so the lattice of submodules
of $Y$ is isomorphic to that of the $\!D\!$-vector-space $f'Y.$
Hence in our proof, we may assume without loss of generality
that $V$ is a division ring $D,$ and $Y$ a $\!D\!$-vector space.

Given proper subspaces $Y_1,\dots,Y_n$ of $Y$ with $n\leq N,$
suppose inductively that for some $m<n$ we have found an element
$y\in Y$ which does not lie in any of $Y_1,\dots,Y_m.$
If $y$ also does not lie in $Y_{m+1},$ we have our next inductive step.
If, on the other hand, $y\in Y_{m+1},$
take any $y'\notin Y_{m+1},$
and consider the $N$ elements $y'+\alpha y,$ as $\alpha$ runs over $D.$
Clearly, none of these lie in $Y_{m+1},$ and it is easy
to check that at most one can lie in each of $Y_1,\dots,Y_m.$
Since they constitute $N\geq n> m$ elements, at least one lies in
none of these spaces, giving the inductive step.

The second statement is likewise a statement about the lattice
of submodules of $Y,$ so in proving it
we can again assume $V$ a division ring.
Recalling that maximal subspaces of $Y$ are kernels
of elements of the dual space $Y^*=\r{Hom}(Y,{_D}D),$ a
right $\!D\!$-vector-space, we can apply to that space the
left-right dual of the first statement, and get the desired result.
\end{proof}

The next corollary applies the two preceding results to the
modules and map of~\eqref{d.setup}.
Note that the statement refers to both
{\em minimal} idempotents and {\em minimal central} idempotents.

\begin{corollary}\label{C.N-strong}
Suppose $h:{_T}A_S \times {_S}B\to {_T}C$ is as in~\eqref{d.setup},
and satisfies~\eqref{d.nondeg},~\eqref{d.b}, and~\eqref{d.c}.

Let $N_T$ be the minimum of the cardinalities of the
division rings $f\,Tf$ as $f$ ranges over the minimal idempotents
of $T,$ if that minimum is finite, or
the symbol $\infty$ if all those division rings are infinite.
Let $d_T$ be the maximum, as $f$ ranges
over the minimal {\em central} idempotents of $T,$ of the number
of minimal central idempotents $e$ of $S$ such that $fA\,e\neq 0.$
Then for each minimal idempotent
$f$ of $T,$ there is at least one minimal idempotent $e$
of $S$ such that $fA\,e$ is left $\!N_T/d_T\!$-strong as a set of
maps $e\,B\to f\,C.$

Likewise, let $N_S$ be the minimum of the cardinalities of the
division rings $eSe$ as $e$ ranges over the minimal idempotents of $S$
if this is finite, or the symbol $\infty$ if that minimum is infinite,
and $d_S$ the maximum, as $e$ ranges
over the minimal {\em central} idempotents of $S,$
of the number of minimal central idempotents $f$ of $T$
such that $fA\,e\neq 0.$
Then for each minimal idempotent
$e$ of $S,$ there is at least one minimal idempotent $f$
of $T$ such that $fA\,e$ is right $\!N_S/d_S\!$-strong as a set of
maps $e\,B\to f\,C.$
\end{corollary}

\begin{proof}
In the situation of the first assertion, we see from
condition~\eqref{d.c} that for each minimal central idempotent
$f$ of $T,$ the set $f\,A$ of morphisms
$f\,a: B\to f\,C$ $(a\in A)$ has the property
that for every minimal $\!f\,Tf\!$-submodule $C_0\subseteq f\,C,$
there is at least one nonzero $f\,a\in f\,A$ which takes $B$ into $C_0.$
Now the minimal central idempotents $e\in S$ sum to $1,$ so there
is at least one such $e\in S$ such that $f\,a\,e\neq 0.$
This map clearly still carries $B$ into $C_0;$ so we conclude that
for every minimal $C_0\subseteq f\,C,$ some nonzero
element of $\bigcup_e (fA\,e-\{0\})$ has image in $C_0,$ where the
union is over the minimal central idempotents $e$ of $S.$
By Lemma~\ref{L.not_union}, this shows that if we write $D$
for the division ring (unique up to isomorphism) given by
$f'Tf'$ where $f'$ is any minimal idempotent of $f\,Tf,$
then $\bigcup_e (fA\,e-\{0\})$ is left $\!\card(D)\!$-strong; hence
left $\!N_T\!$-strong.

Now by definition of $d_T,$ there are at most $d_T$ minimal
central idempotents $e$ such that $fA\,e\neq 0;$ so
$\bigcup_e (fA\,e-\{0\})$ involves at most $d_T$ nonempty sets
$fA\,e-\{0\}.$
Hence by Lemma~\ref{L.strong_union}, at least one of these
sets is left $\!N_T/d_T\!$-strong, as claimed.

The second assertion is proved in the analogous fashion.
\end{proof}

The above corollary will enable us to prove our
desired generalizations of~\eqref{d.cm_dim} and~\eqref{d.cm_local_dim}
{\em unless} one or more of the division rings
$eSe$ and $f\,Tf$ are finite fields of small cardinality.
(If such a field occurs, $N_T/d_T$ or $N_S/d_S$
may be too small for our arguments to work.)
It is curious that a similar condition in my original proof
of~\eqref{d.cm_dim} and~\eqref{d.cm_local_dim} was
eliminated by de~Seguins Pazzis's argument;
but we shall see by example, in \S\ref{S.misc},
that the corresponding condition in the present
situation cannot be dropped.

While these cardinality conditions are not very restrictive,
and are needed for the result to hold, we come now to an
embarrassingly restrictive condition, needed for the proofs
of the results of the next section,
though I have no example showing that those results fail without it.
The condition is awkward to state in maximum generality.
A fairly natural special case is the hypothesis that
in~\eqref{d.setup},
\begin{equation}\begin{minipage}[c]{35pc}\label{d.small_algcl}
The semisimple Artinian rings $S$ and $T$ are finite-dimensional
algebras over a common algebraically closed field $k,$
and the induced actions of $k$ on the two sides of the bimodule
$_T A_S$ are the same.
\end{minipage}\end{equation}
Actually, we need the assumption that $k$ is algebraically
closed only to make the simple factors of $S$ and $T$
full matrix algebras over $k;$ so we can instead put
that assumption on the table, as the condition
\begin{equation}\begin{minipage}[c]{35pc}\label{d.small_mxs}
For some field $k,$ each of the semisimple Artinian rings $S$ and $T$
is a direct product of full matrix algebras over $k,$ and
the induced actions of $k$ on the two sides of the bimodule
$_T A_S$ are the same.
\end{minipage}\end{equation}
A condition that is still more general (as we shall show in the
next lemma), and will suffice for our purposes, is
\begin{equation}\begin{minipage}[c]{35pc}\label{d.small_both}
If $a\in A$ is a nonzero element whose image $a\,B$ is contained
in a simple submodule of $C,$ then there exists nonzero $a'\in TaS$
whose kernel contains a maximal submodule of $B;$ and likewise
if $a\in A$ is a nonzero element whose kernel contains
a maximal submodule of $B,$ then there exists nonzero $a'\in TaS$
whose image is contained in a simple submodule of $C.$
\end{minipage}\end{equation}
In fact, we can make do with the following still weaker
(though still wordier) condition.
\begin{equation}\begin{minipage}[c]{35pc}\label{d.small_either}
For each minimal central idempotent $e\in S$ and
minimal central idempotent $f\in T$
such that $f\,A\,e\neq\{0\},$ it is {\em either} true that
for every $a\in fA\,e$ such that the induced map $e\,B\to f\,C$ has
image in a simple submodule of $f\,C,$ some nonzero
$a'\in TaS$ has kernel containing a maximal submodule of $B,$
{\em or} that for every $a\in fA\,e$ such that the induced map
$e\,B\to f\,C$ has kernel containing a maximal submodule of $B,$
some nonzero $a'\in TaS$ has image in a simple submodule of $f\,C.$
(But which of these is true may vary with the pair $(f,e).)$
\end{minipage}\end{equation}

Let us note the implications among these conditions.

\begin{lemma}\label{L.conditions}
For $h:{_T}A_S \times {_S}B\to {_T}C$ a bilinear map as
in~\eqref{d.setup} which
satisfies~\eqref{d.nondeg},
one has the implications \eqref{d.small_algcl}$\!\implies\!$%
\eqref{d.small_mxs}$\!\implies\!$%
\eqref{d.small_both}$\!\implies\!$\eqref{d.small_either}.
\end{lemma}

\begin{proof}
\eqref{d.small_algcl}$\implies$\eqref{d.small_mxs} follows from the
standard description of the structures of semisimple Artin rings,
and \eqref{d.small_both}$\!\implies\!$\eqref{d.small_either} is clear.
Let us prove \eqref{d.small_mxs}$\implies$\eqref{d.small_both}.

Suppose as in~\eqref{d.small_both} that $a\neq 0,$ and $a\,B$ is
contained in a simple submodule $C_0\subseteq C.$
Since the identity elements of $S$ and $T$ are sums of
minimal central idempotents, we can find such idempotents
$e\in S$ and $f\in T$ such that $f\,a\,e\neq 0;$ and we will
have $f\,a\,e\,B\subseteq f\,a\,B\subseteq f\,C_0\subseteq C_0.$
Hence, replacing $a$ by some $a'=f\,a\,e$ if necessary,
we may assume without loss of generality that $a\in fA\,e$
for such a pair of idempotents.

Now by assumption, $eSe$ and $f\,Tf$ have the forms
$\r{Matr}_{m,m}(k)$ and $\r{Matr}_{n,n}(k)$ for
some positive integers $m$ and $n.$
Identifying them with these matrix rings,
it is easy to verify that there exist finite-dimensional
$\!k\!$-vector-spaces $B',$ $C'$ such that we can
identify $B$ and $C$ with the spaces $B'^m$ and $C'^n$ of
column vectors of elements of $B'$ and $C',$ made into modules over
$S=\r{Matr}_{m,m}(k)$ and $T=\r{Matr}_{n,n}(k)$ in the natural
way; that $A$ can then be identified
with $\r{Matr}_{n,m}(A')$ where $A'$ is a $\!k\!$-vector-space
of $\!k\!$-linear maps $B'\to C',$ and finally, that the simple
submodule $C_0\subseteq f\,C$ will have the form $C_0'^n$ for
some $\!1\!$-dimensional subspace $C'_0\subseteq C'.$
(Explicitly, letting $e'$ and $f'$ denote minimal idempotents
of $S$ and $T,$ say those given by the matrix units $e_{11}$
of their matrix representations, we can take $B'=e'B,$
$C'=f'C,$ $C'_0=f'C_0,$ and $A'=f'A\,e'.)$

Thus, our element $a\in fA\,e$ will be an $n\times m$ matrix of linear
maps $e\,B'\to f\,C',$ each having range in $C'_0.$
We can now choose $\bar{a}\in TaS-\{0\}$
to be nonzero and have all components
in some $\!1\!$-dimensional subspace $k\,a'\subseteq A'.$
(E.g., we can let $\bar{a}$ be a product $e_{1,i}\,a\,e_{j,1}$ such
that the $(i,j)$ component of $a$ is nonzero.)
Since $a': B'\to C'$ has range in the $\!1\!$-dimensional
subspace $C'_0$ of $C',$ it is a rank-$\!1\!$ $\!k\!$-linear map, hence
has kernel $B'_0\subseteq B$ of codimension $1.$
Hence $\bar{a},$ having all components in $k\,a',$
will have
kernel containing the maximal proper submodule
$(B'_0)^m\subseteq B,$ as required.

The second assertion of~\eqref{d.small_both} is proved similarly.
\end{proof}

\section{Minimal faithful modules over left Artin rings}\label{S.lt_art}

Recall that if $G$ is a finite graph, its {\em Euler
characteristic} $\chi(G)$ is the number of vertices
of $G$ minus the number of edges, an integer which may be
positive, negative or zero.
Recall also that a {\em bipartite} graph is a graph whose vertex-set is
given as the union of two specified sets, such that
every edge connects a member of one set with a member of the other.
We shall call those sets (nonstandardly) the {\em left} and {\em right}
vertex-sets of~$G.$

The hard work of this section comes right at the
beginning: proving the following
noncommutative analog of Proposition~\ref{P.cm_bilin},
or more precisely, of Corollary~\ref{C.cm}.

\begin{proposition}\label{P.main}
Suppose $h:{_T}A_S \times {_S}B\to {_T}C$ is a bilinear map as
in~\eqref{d.setup}, which satisfies~\eqref{d.nondeg},
\eqref{d.b}, \eqref{d.c} and~\eqref{d.small_either}.
For notational convenience we shall assume $S$ and $T$ disjoint.

Let $N_S,$ $d_S,$ $N_T$ and $d_T$ be defined as
in Corollary~\ref{C.N-strong}.
Further, let $l_S$ denote the maximum of the values $\lt_{eSe}(e\,B)$
as $e$ ranges over the minimal central idempotents $e\in S,$ and
$l_T$ the maximum of $\lt_{f\,Tf}(f\,C)$ as $f$ ranges
over the minimal central idempotents $f\in T;$ and assume that
\begin{equation}\begin{minipage}[c]{35pc}\label{d.cardD}
$N_T\ \geq\ d_T\,l_S,$\quad and\quad $N_S\ \geq\ d_S\,l_T.$
\end{minipage}\end{equation}

Finally, let $G$ be the bipartite graph whose right vertex-set is the
set of minimal central idempotents $e\in S$ satisfying $e\,B\neq 0$
\textup{(}equivalently, $A\,e\neq 0),$
whose left vertex-set is the
set of minimal central idempotents $f\in T$ satisfying $f\,C\neq 0$
\textup{(}equivalently, $f\,A\neq 0),$ and
such that two such vertices $e,$ $f$ are connected
by an edge $(f,e)$ if and only if $fA\,e\neq\{0\}.$

Then
\begin{equation}\begin{minipage}[c]{35pc}\label{d.P.main}
$\lt({_S}B)\ +\ \lt({_T}C)\ \leq\ \lt({_T}A_S)\ +\ \chi(G).$
\end{minipage}\end{equation}
\end{proposition}

\begin{proof}
The parenthetical equivalences in the definition of $G$ follow
from~\eqref{d.nondeg},~\eqref{d.b} and~\eqref{d.c}.
Combining Corollary~\ref{C.N-strong} with our
hypothesis~\eqref{d.cardD}, we find that
\begin{equation}\begin{minipage}[c]{35pc}\label{d.strong}
For each minimal idempotent
$f$ of $T,$ there is at least one minimal idempotent $e$
of $S$ such that $fA\,e$ is left $\!l_S\!$-strong as a set of
maps $e\,B\to f\,C;$ and for each minimal idempotent
$e$ of $S,$ there is at least one minimal idempotent $f$
of $T$ such that $fA\,e$ is right $\!l_T\!$-strong as a set of
maps $e\,B\to f\,C.$
\end{minipage}\end{equation}

We shall now perform a series of reductions and decompositions on our
system ${_T}A_S \times {_S}B\to {_T}C,$ verifying at each stage that
if the inequality corresponding to~\eqref{d.P.main} holds
for our simplified system(s), then it also holds for the
original system; and, finally, we shall establish that
inequality for the very simple sorts of system we end up with.

In preparation, let us harness~\eqref{d.strong}
by choosing, arbitrarily, for each minimal
central idempotent $f\in T,$ {\em one} minimal central idempotent
$e\in S$ such that $fA\,e$ is left $\!l_S\!$-strong, and call
$(f,e)$ the {\em left-marked} edge of the graph $G$ associated
with the vertex $f;$ and similarly, for each minimal
central idempotent $e\in S,$ let us choose a minimal central idempotent
$f\in T$ such that $fA\,e$ is right $\!l_T\!$-strong,
and call $(f,e)$ the {\em right-marked} edge of $G$ associated
with the vertex $e.$
Some edges may be both right- and left-marked
(for their respective right and left vertices).

We begin our
reductions by considering any edge $(f,e)\in G$ which is neither
right- nor left-marked, and seeing what happens if we drop the
summand $fA\,e$ from $A;$
i.e., replace $A$ with $(1-f)A + A(1-e);$ and thus drop
the edge $(f,e)$ from $G,$ leaving the rest of our system unchanged.
Because $(f,e)$ is neither right nor left marked,
condition~\eqref{d.strong} has not been lost.
(The constant $d_S$ and/or $d_T$ may have decreased by $1,$
but we don't have to think about this, because our use of
these constants was only to obtain~\eqref{d.strong},
which has been preserved.)
The removal of $fA\,e$ has no effect on the left-hand
side of~\eqref{d.P.main}, while on the right-hand side,
it decreases $\lt({_T}A_S)$ by
$\lt({_T}(fA\,e)_S)\geq 1,$ and increases $\chi(G)$ by $1.$
Hence if the new system satisfies~\eqref{d.P.main}, then
the original system, whose right-hand side is $\geq$ that of
the new system, also did.

Hence by induction, the task of proving~\eqref{d.P.main} is
reduced to the case where {\em every} edge of $G$ is left
and/or right marked.

Suppose, next, that there is some left vertex $f$ of $G$ such
that the only edge adjacent to $f$ is its associated left-marked
edge, say $(f,e),$ and such that this is {\em not} also right-marked,
and consider what happens if we remove both the vertex $f$ and
the edge $(f,e);$ i.e., replace $C$ by $(1-f)C,$ and $A$ by $(1-f)A.$

Clearly, the {\em remaining} vertices and edges continue to
witness condition~\eqref{d.strong}.
Our new system also has the same Euler characteristic as the old one,
since just one vertex and one edge have been removed from the graph.

To see how~\eqref{d.P.main} is affected, let $d=\lt({_T}f\,C).$
I claim that we can find
elements $a_1,\dots,a_d\in fA\,e$ such that the submodules
$T\,a_i\,B\subseteq f\,C$ are simple and their sum is direct.
Indeed, assuming we have constructed $a_1,\dots,a_j$ with $j<d,$
the submodule $\bigoplus_{i\leq j} T\,a_i\,B\subseteq f\,C$
will have length $\leq j< d=\lt({_T}f\,C),$ so
the fact that $fA\,e$ is left $\!1\!$-strong (a weak consequence
of the fact that $(f,e)$ is left marked, hence that $fA\,e$ is
left $\!l_S\!$-strong) allows us to find $a_{j+1}$
with $T\,a_{j+1}\,B$ simple and not contained
in that proper submodule; and since it is simple,
its sum with that submodule is direct.
Once we have $a_1,\dots,a_d,$ we see that
the sub-bimodules $\sum_{i\leq j} T\,a_i\,S\subseteq A$ $(j\leq d)$
form a chain of length $d.$
Hence $\lt({_T}(fA\,e)_S)\geq d=\lt({_T}f\,C),$ from which
we see that replacing $A$ with $(1-f)A$ decreases
the right-hand side of~\eqref{d.P.main}
by at least as much as replacing $C$ with $(1-f)C$ decreases
the left-hand side.
So again, if the new system satisfies~\eqref{d.P.main},
so did the old one.

Similarly, if $G$ has some right vertex $e$ such
that the only edge adjacent to $e$ is its associated right-marked
edge $(f,e),$ and this edge is not also left-marked,
we find that we can remove $fA\,e$ from $A,$
and $e\,B$ from $B,$ and that if the resulting system
satisfies~\eqref{d.P.main}, so does our original system.
The proof is the same, except that where above we used
an increasing family of submodules
$\bigoplus_{i\leq j} T\,a_i\,B\subseteq f\,C,$
we now use a decreasing family of submodules
$\bigcap_{i\leq j} \r{ker}(a_i\,S)\subseteq e\,B.$

Repeating these two kinds of
reductions until no more instances are possible,
we are left with a system in which every
vertex not only hosts its own marked edge, but also hosts
the marked edge of {\em at least one} other vertex (which may or
may not be the same as its own marked edge).
By counting vertices, we immediately see that in the preceding
sentence, ``{\em at least one}'' can be replaced by
``{\em exactly one}''.
From this, it is easy to see that each connected component of $G$
is now either a loop of even length $>2,$ or a graph having just
two vertices, and a single edge which is marked for both of these.

It follows that our
system ${_T}A_S \times {_S}B\to {_T}C$
decomposes into a direct sum of subsystems corresponding to
those connected components, and that~\eqref{d.P.main}
will be the sum of the corresponding inequalities for those components.
Hence, it will suffice to prove~\eqref{d.P.main}
in the two cases where $G$ is a loop, and where $G$ has
just a single edge.
The $l_S$ and $l_T$ for each such system are $\leq$
the $l_S$ and $l_T$ for our original system, so~\eqref{d.strong} will
hold for these systems, because it held for the original system.

If $G$ is a loop, then each edge is marked for only
one vertex, and the proof is quick:
The Euler characteristic $\chi(G)$ is zero, while for each
vertex, the bimodule corresponding to its marked edge
has at least the length of the module corresponding to
that vertex, by the same ``$\sum_{i\leq j} T\,a_i\,S$''
arguments used in the preceding reduction.
Summing these inequalities over the vertices, we have~\eqref{d.P.main}.

We are left with the case where $G$ has just one
right vertex, $e,$ one left vertex, $f,$ and the single edge $(f,e).$
Thus, $B=eB$ and $C=fC.$
Let
\begin{equation}\begin{minipage}[c]{35pc}\label{d.m,n}
$m\ =\ \lt({_S}B),\qquad n\ =\ \lt({_T}C).$
\end{minipage}\end{equation}
Note that the fact that $(f,e)$ is both left- and right-marked
tells us that $A$ is both left $\!l_S\!$-strong and
right $\!l_T\!$-strong.

It is now that we will use our hypothesis~\eqref{d.small_either}.
Let us begin by assuming the second of the two alternatives
it offers, which in this situation says that
whenever an $a\in A$ has kernel containing a
maximal submodule of $B,$ then some nonzero element of
$TaS$ has image in a simple submodule of $C.$
Under this assumption, we begin by constructing,
for $m$ as in~\eqref{d.m,n}, elements $a_1,\dots,a_m\in A$ such that
\begin{equation}\begin{minipage}[c]{35pc}\label{d.a1...am}
$T\,a_1\,S,\dots,T\,a_m\,S$
have for kernels maximal submodules of $B,$ none of which contains
the intersection of the kernels of the others, and each $T\,a_i\,S$
has for image a simple submodule of~$C.$
\end{minipage}\end{equation}
To see that we can do this, suppose inductively that we have
constructed $i<m$ elements
$a_1,\dots,a_i$ with these properties.
Since $i<m=\lt({_S}B),$ the intersection of the kernels
of $T\,a_1\,S,\dots,T\,a_i\,S$ is a nonzero submodule
of $B,$ hence since $A,$ being right $\!l_T\!$-strong, is
in particular right $\!1\!$-strong,
we can find $a\in A$ such that the kernel of $a\,S$ is
a maximal submodule $B_0\subseteq B$ and does
not contain that intersection.
Now any $a'\in T\,a_i\,S-\{0\}$ will still annihilate $B_0;$
hence, if it is nonzero, $T\,a'\,S$ must have exactly that annihilator.
By our assumption from~\eqref{d.small_either}, we can
find such a nonzero $a'$ which has image in a simple submodule of $C.$
Taking for $a_{i+1}$ this element, we have our
desired inductive step; for the intersection
of the kernels of $T\,a_1\,S,\dots,T\,a_{i+1}S$ has co-length $i+1$
in $B,$ hence none can contain the intersection of the kernels
of the others.
Thus, we get~\eqref{d.a1...am}.

I claim next that we can find $n-1$ more
elements, $a'_1,\dots,a'_{n-1}\in A,$ where each $a'_j$ satisfies
\begin{equation}\begin{minipage}[c]{35pc}\label{d.tricky}
$T\,a'_j\,B$ is a simple submodule of $C,$ and
for each $i=1,\dots,m,$ we have\\
$T\,a'_j\,B\ \not\subseteq\ T\,a_i\,B+T\,a'_1\,B+\dots+T\,a'_{j-1}\,B.$
\end{minipage}\end{equation}
To see this, suppose inductively that
$a'_1,\dots,a'_{j-1}$ have been chosen.
For each $i\leq m,$ both
$T\,a_i\,B$ and each of $T\,a'_1\,B,\dots, T\,a'_{j-1}\,B$
are simple submodules of $C,$ and there are $1+(j-1)=j<n$ of
these, so their sum is a proper submodule.
As $i$ ranges from $1$ to $m$ we get $m=l_S$ such proper submodules;
so as $A$ is left $\!l_S\!$-strong, our $a'_j$ can be chosen
to satisfy~\eqref{d.tricky}.

Let us now show that as we add up successively the subbimodules
$T\,a_1\,S+\dots+T\,a_m\,S+T\,a'_1\,S + \dots + T\,a'_{n-1}\,S$ of $A,$
we get a chain of length $m+n-1.$
The first $m$ steps are distinct
by comparison of their kernels in $B,$ in view of~\eqref{d.a1...am}.
If we had equality somewhere in the next $n-1$ steps, this
would mean that for some $j<n$ we would have
\begin{equation}\begin{minipage}[c]{35pc}\label{d.a'_j}
$T\,a'_j\,S\ \subseteq
\ T\,a_1\,S+\dots+T\,a_m\,S+T\,a'_1\,S+\dots+T\,a'_{j-1}\,S.$
\end{minipage}\end{equation}

To get a contradiction, let us, for $i=1,\dots,m,$
write $B_i$ for the intersection of the kernels of
all the $T\,a_j\,S$ other than $T\,a_i\,S.$
By~\eqref{d.a1...am}, each $B_i$ has co-length
$m-1,$ hence is a simple submodule of $B,$ and
no $B_i$ is contained in the sum of the others; so $\sum_i B_i=B.$
Hence, some $B_i$ is not in the kernel of $T a'_j S;$
let us choose such a $B_i$ and apply~\eqref{d.a'_j} to it.
We get, on the left, $T a'_j B_i,$ which by our choice of $i$
is nonzero.
Since it is contained in $T a'_j B,$ which
by~\eqref{d.tricky} is simple, it must be equal thereto.
On the right, by definition of $B_i,$
we loose all of the first $m$ terms other than the $\!i\!$-th.
Replacing the remaining occurrences of $B_i$ on the right by the larger
module $B,$ we get
\begin{equation}\begin{minipage}[c]{35pc}\label{d.contradiction}
$T\,a'_j B\ \subseteq
\ T\,a_i\,B\ +\ T\,a'_1\,B\ +\ \dots\ +\ T\,a'_{j-1}\,B.$
\end{minipage}\end{equation}
But this is one of the relations we chose $a'_j$ to avoid
in~\eqref{d.tricky}.
This contradiction proves that our chain of submodules of $A$
is strictly increasing, hence that $A$ has length at least
$m+n-1=\lt({_S}B)+\lt({_T}C)-\chi(G),$ as required.

If we are in the other case of~\eqref{d.small_either}, we operate
dually, and first obtain $n$ elements $a_1,\dots,a_n$ of $A$ which
determine independent simple submodules of $C,$ and
each of which has kernel containing a maximal submodule
of $B,$ then choose $m-1$ more elements $a'_1,\dots,a'_{m-1},$ such
that the submodule of $B$ which each determines is maximal,
and does not contain the intersection of the submodules
determined by the proceeding members of that list,
intersected with the submodule determined by any one of the $a_i.$
\end{proof}

We can now get our main result.
Since I don't see how to turn~\eqref{d.small_both}
or~\eqref{d.small_either} into a condition on the ring $R,$
I will only give the hypothesis corresponding to~\eqref{d.small_mxs}.

Condition~\eqref{d.cardD} would put a finite lower
bound on the size of $k;$ but since this
bound would involve both the structure of $R$ and that of $M,$
and we would like our restrictions on $M$ to be stated
in terms of the structure of $R,$
I will achieve this simply by requiring $k$ to be infinite.

The two rings $S$ and $T$ of the preceding development
are the same ring $R/J(R)$ in the application below, so we shall
distinguish the corresponding sorts of vertices of our graph with
subscripts ``left'' and ``right''.
We shall regard $\soc(R)$ (the $\!2\!$-sided socle of $R,$ whose
definition we recalled in the paragraph before Theorem~\ref{T.TG})
as a bimodule over $R/J(R);$ as such it will play the role of
the $A$ of Proposition~\ref{P.main}.

\begin{theorem}\label{T.main}
Suppose $k$ is an infinite field and $R$ an overring of $k,$
such that $R$ is finite-dimensional as a left $\!k\!$-vector-space,
$k$ has central image in $R/J(R)$ and centralizes $\soc(R),$
and $R/J(R)$ is a direct product of full matrix
rings over $k$ \textup{(}this last condition being automatic
if $k$ is algebraically closed\textup{)}.

Let $G$ be the bipartite graph whose left vertex-set
consists of symbols $f_{\r{left}}$ for all minimal
central idempotents $f\in R/J(R)$ such that $f\,\soc(R)\neq\{0\},$
whose right vertex-set
consists of symbols $e_{\r{right}}$ for all minimal
central idempotents $e\in R/J(R)$ such that $\soc(R)\,e\neq\{0\},$
and whose edge-set is
$\{(f_\r{left},\,e_\r{right})\mid f\soc(R)\,e\neq\{0\}\,\}.$

Then for any a faithful left $\!R\!$-module $M$ such that
no proper submodule
or proper homomorphic image of $M$ is faithful, we have
\begin{equation}\begin{minipage}[c]{35pc}\label{d.T.main}
$\lt(M/J(R)M)\ +\ \lt(\soc(M))\ \leq\ \lt(\soc(R))\ +\ \chi(G),$
\end{minipage}\end{equation}
where the two lengths on the left are as left $\!R\!$-modules,
while the length on the right is as an $\!(R,R)\!$-bimodule.
\end{theorem}

\begin{proof}[Sketch of proof]
Under the action of $R$ on $M,$ elements of
$\soc(R)$ annihilate $J(R)M$ and have
image in $\soc(M);$ so that action induces a balanced bilinear map
of $\!R\!$-modules and bimodules, $\soc(R)\times M/J(R)M\to \soc(M).$
Since all four $\!R\!$-module structures involved (the
three left module structures, and the right module
structure of $\soc(R))$ are annihilated by $J(R),$
that induced operation is a map of the form~\eqref{d.setup},
with $R/J(R)$ in the roles of both $S$ and $T.$
Since $M$ is faithful, this map satisfies~\eqref{d.nondeg}, while
as in the proof of Theorem~\ref{T.TG},
the minimality assumptions on $M$
give~\eqref{d.b} and~\eqref{d.c}.
Since~\eqref{d.small_mxs} implies~\eqref{d.small_either},
we can apply the preceding proposition, getting~\eqref{d.T.main}.
\end{proof}

\section{Examples, remarks, and questions}\label{S.misc}
\subsection{Counterexamples over small fields}\label{SS.Z2}
To show that Proposition~\ref{P.main} can fail if
condition~\eqref{d.cardD} (the requirement that our division rings
not be too small) is dropped, let $k$ be the field of
$2$ elements, let $A=B=k\times k\times k$
and $C=k\times k$ (as abelian groups for the moment),
and define $h:A\times B\to C$ by
$h((\alpha_1,\alpha_2,\alpha_3),(\beta_1,\beta_2,\beta_3))
=\alpha_1\beta_1(1,0)+ \alpha_2\beta_2(0,1)+ \alpha_3\beta_3(1,1).$
Letting $S=k\times k\times k$ and $T=k,$ and defining
the module structures of $_T A_S,$ $_S B$ and $_T C$ in the
obvious ways, we see that $h$ is indeed a balanced bilinear map,
i.e., satisfies~\eqref{d.setup}.

This map also clearly satisfies~\eqref{d.small_mxs}
and~\eqref{d.nondeg}, and it is not hard to
verify~\eqref{d.b} and~\eqref{d.c} as well.
However,~\eqref{d.cardD} fails, since $N_T=2,$ $d_T=3,$ $l_S=1.$
And in fact, the conclusion~\eqref{d.P.main} fails: the left-hand side
of that inequality is $3+2,$ while the right-hand side is $3+1.$
If we examine the steps of our proof of Proposition~\ref{P.main}
in this case, we see that conditions~\eqref{d.b} and~\eqref{d.c} make
$A,$ regarded as a family of maps $B\to C,$ left $\!2\!$-strong, but
of its three components $A\,e_i: e_i\,B\to C,$
none is $\!1\!$-strong.
We can, as in the proof of that proposition, snip two leaves off
$G$ without changing the numerical relationship between
the two sides of~\eqref{d.P.main}; but the remaining system, say
$A\,e_1: e_1\,B\to C,$ does not satisfy ~\eqref{d.P.main}.

More generally, over any finite field $k,$ say of $q$ elements,
one can get a similar example by taking $S=A=B=k^{q+1},$ $T=k,$ and
$C=k^2,$ and letting the $q+1$ components $A\,e_i: B\to C$ have
for images the $q+1$ one-dimensional subspaces of $C.$
(The reader who has worked through the above
$q=2$ case should not find it hard to supply the details
for this generalization.)
Still more generally, if we take $C$ to be $\!d\!$-dimensional
$(d\geq 2),$ we can make $S=A=B=k^{(q^d-1)/(q-1)},$ and let the
natural basis of $A$ act by maps having for images the
$(q^d-1)/(q-1)$ one-dimensional subspaces of $C.$

We can adapt these constructions to get examples showing that
for small $k$ Theorem~\ref{T.main} likewise fails;
but since in that case $S$ and $T$ must be the same, a bit
of adjustment is needed.
We can keep $S$ as in those constructions, but let $T=S,$
giving it actions on $C$ and $A$ under which all but one of its
minimal idempotents annihilate those objects.

Let me describe in concrete terms an $R$ and $M$ based, in this way,
on our initial example where $k$ is the field of $2$ elements.
Our $R$ will be the ring of all $4\times 4$
matrices over this $k$ with support in the union of the first row
and the main diagonal, and whose $(1,1)$ and $(2,2)$ entries are equal.
To obtain $M,$ we start with the left $\!R\!$-submodule of
$\r{Matr}_{4,3}(k)$ spanned as a $\!k\!$-vector-space by
$\{e_{11},\,e_{12},\,e_{13},\,e_{21},\,e_{32},\,e_{43}\},$
and divide out by the subspace spanned by $e_{11}+e_{12}+e_{13}.$
We find that $J(R)M=\soc(M)=$ the $\!2\!$-dimensional space
spanned by $e_{11}$ and $e_{12};$
that $M$ is a faithful $\!R\!$-module such that no proper
submodule or homomorphic image of $M$ is faithful, but that
$M$ does not satisfy~\eqref{d.T.main}, which for this case
would say $3+2\leq 3+1.$

\subsection{An example not satisfying~\eqref{d.small_either}}\label{SS.small}
I will give here an example of a system~\eqref{d.setup}
arising as the map $\soc(R)\times M/J(M)\to\soc(M)$ for a
faithful module $M$ over an Artinian local ring $R,$ which
does not satisfy~\eqref{d.small_either}, or our
conclusion~\eqref{d.P.main}.
In fact, it does not satisfy~\eqref{d.b} or~\eqref{d.c}
either; but it may give some insight into how things can
differ from the situation analyzed in the preceding sections.

Let $K$  be the field $\Q(2^{1/3}),$ let $F = K(\omega),$ where
$\omega$ is a primitive cube root of unity, and let $\sigma$ be the
automorphism of $F$ of order 3  which fixes $\omega,$ and
takes $2^{1/3}$ to $\omega\,2^{1/3}.$
Let $\r{tr}=\r{tr}_{F/K}: F\to K$ be the trace operation.
(Thus, $\r{tr}$ is $\!K\!$-linear, but $\sigma$ is not.)

Let $M$ be the $\!\Q\!$-vector-space $F^2,$ and let us define the
right shift operation $M\to M,$
\begin{equation}\begin{minipage}[c]{35pc}\label{d.s}
$s: (a,b)\mapsto (0,a).$
\end{minipage}\end{equation}
We shall understand $\r{tr},$ $\sigma$ and each element of $F$
to act on $M$ componentwise.
(So if $a\in F,$ the symbol $a$ will also represent
the operation of componentwise multiplication of elements of $M$
by $a,$ which does not in general commute with either $\r{tr}$
or $\sigma.)$

We now define two operations on $M,$
\begin{equation}\begin{minipage}[c]{35pc}\label{d.x,y}
$x = \r{tr}\ \sigma\ s:\ (a,b)\mapsto (0,\r{tr}(\sigma(a))),$\\
\hspace*{.03em}%
$y = \sigma\ \r{tr}\ s:\ (a,b)\mapsto (0,\sigma(\r{tr}(a))).$
\end{minipage}\end{equation}

From the fact that $[F:K]=2,$ it is easy to see
that $K\,\sigma(K)=F=K\,\sigma^{-1}(K).$
Combining this with the fact that $\r{tr}$ and $s$ commute with the
action of elements of $K,$ we see that
\begin{equation}\begin{minipage}[c]{35pc}\label{d.xF,Fy}
$K\,x\,K = K\,\r{tr}\,\sigma\,s\,K = \r{tr}\,K\,\sigma\,s\,K =
\r{tr}\ \sigma\ (\sigma^{-1}(K))\,s\,K =
\r{tr}\ \sigma\ s\ (\sigma^{-1}(K)\,K) = x\,F,$\\
\hspace*{.03em}%
$K\,y\,K = K\,\sigma\,\r{tr}\,s\,K = K\,\sigma\,\r{tr}\,K\,s =
K\,\sigma\,K\,\r{tr}\,s = (K\,\sigma(K))\,\sigma\,\r{tr}\,s = F\,y.$
\end{minipage}\end{equation}
These calculation show in particular that the bimodule
operations of the $\!(K,K)\!$-bimodules
spanned by each of $x$ and $y$ contain the operations of a
$\!1\!$-dimensional $\!F\!$-module; so each of these
bimodules is simple.

Now let $R$ be the ring of $\!\Q\!$-vector-space endomorphisms
of $M$ generated by the actions of the elements of $K,$
together with the two endomorphisms $x$ and $y.$
We see that
\begin{equation}\begin{minipage}[c]{35pc}\label{d.R}
$R=K+K\,x\,K+K\,y\,K=K+x\,F+F\,y,$\quad and \quad
$J(R)=\soc(R) = x\,F + F\,y.$
\end{minipage}\end{equation}

The subideals of $\soc(R)$ generated by $x$ and by $y$ are
isomorphic to one another as bimodules over $R/J(R)\cong K,$ namely,
each is isomorphic to the $\!(K,K)\!$-bimodule $F\,\sigma=\sigma\,F.$
But as systems of maps $M/J(R)M\to \soc(M)$ they behave differently:
$xF$ has range in $(0,K\,\r{tr}(F))=(0,K)\neq (0,F),$
but has trivial kernel; $Fy,$ dually, has all of $(0,F)$ as
range, but, identifying $M/J(R)M$ with $F,$ its kernel
is the nontrivial $\!K\!$-subspace $\r{ker}(\r{tr}).$

\subsection{An approach that probably doesn't go anywhere}\label{SS.extend_scalars}
Theorem~\ref{T.main} does not cover the case where $R$ is a general
finite-dimensional algebra $R$ over an infinite field $k,$
since in that situation, $R/J(R)$ need not be a direct product of
matrix rings.
But it is natural to ask: Can't we start with such an
$R$ and an $\!R\!$-module $M,$
extend scalars to the algebraic closure of $k,$
and apply Theorem~\ref{T.main} to the resulting structures?

The trouble is that the relevant properties of $R$ and $M$
may not carry over under this change of base field.
I do not know whether we can expect the minimality conditions
on $M$ to carry over; but the set of minimal central idempotents
of $R$ can certainly grow under such an extension.
So I do not see how anything can be achieved in this way.

\subsection{Some ways our results {\em can} be strengthened}\label{SS.sum_lengths}
One step in our proof of Proposition~\ref{P.main} was noticeably
wasteful.
When we dropped all edges that were neither
left nor right marked, we counted each of them as contributing
``at least $1\!$'' to $\lt({_T}A_S).$
But depending on what we know
about the structure of $A,$ we may be able to raise this estimate.
A difficulty is that knowing the structure of $A$ doesn't
tell us {\em which} edges will be dropped, nor even exactly how many:
the answers may depend on $B$ and $C.$

However, we do know that dropping all unmarked edges must result in
a graph in which there are at least as many vertices as edges
(possibly more, if some edges are marked for both their vertices);
so if the graph $G$ determined by $A$ has more edges than
vertices, i.e., has $\chi(G)<0,$ then
at least $-\chi(G)$ edges $(f,e)$ must be dropped.
Thus, if we have in front of us a list of the lengths of all
the nonzero bimodules $fA\,e$ (with repetitions
shown), then we can say that when
we delete $-\chi(G)$ such bimodules, the sum of their
lengths must be at least the sum of the $-\chi(G)$ smallest
elements on that list.
Deleting $-\chi(G)$ edges will leave us a system whose graph has
Euler characteristic~$0,$ to which we can apply
Proposition~\ref{P.main} as it stands.
Consequently, we have

\begin{corollary}\label{C.better}
In the context of Proposition~\ref{P.main}, let
the family \textup{(}set with multiplicity\textup{)}
of lengths of nonzero subbimodules $fA\,e$ of $A,$
listed in ascending order, be $d_1\leq d_2\leq\dots\,.$
Then if $\chi(G)<0,$
the term $+\chi(G)$ in~\eqref{d.P.main} can be improved
\textup{(}decreased\textup{)} to
\begin{equation}\begin{minipage}[c]{35pc}\label{d.better}
$-d_1-d_2-\dots -d_{-\chi(G)}.$
\end{minipage}\end{equation}
The same applies, mutatis mutandis, to
the inequality~\eqref{d.T.main} of Theorem~\ref{T.main}.\qed
\end{corollary}

One can do still a bit better, using the fact that we don't
delete as unmarked {\em all} the edges adjacent to any vertex.
Consequently, if the lengths of the bimodules
corresponding to edges adjacent to
a certain vertex all lie among the first $-\chi(G)$ terms of the
above sequence $d_1\leq d_2\leq\dots\,,$ then
in~\eqref{d.better} we can skip the largest of these lengths,
and instead throw in $d_{-\chi(G)+1},$ which may be larger, at the end
of the list; and iterate this process for other vertices.

One can also strengthen Theorem~\ref{T.main} by weakening the
assumption that $R$ contains $k,$ to say that
$k$ is the residue field of a local ring contained in $R,$
which again becomes central in $R/J(R)$ and centralizes $\soc(R).$

I have not put these observations into my formal statement
of Theorem~\ref{T.main}, because I feel that the more urgent
task is to see whether one can generalize Proposition~\ref{P.main}
to avoid or weaken the awkward restriction~\eqref{d.small_either};
and that if one can, some of these generalizations might be
embraced by broader, more easily stated results.

\subsection{Sketch of a more elaborate version of the $\!N\!$-strong condition}\label{SS.better_N-str}
The reader may have noticed that in the proof of
Proposition~\ref{P.main}, the final case, where $G$ has
just one edge, both right and left marked,
is roughly the situation of Proposition~\ref{P.cm_bilin}
(as reformulated in Corollary~\ref{C.cm}), but the proof
is different from that of the earlier proposition.
The reason is that the information we have available at that
point in the proof of Proposition~\ref{P.main}, that $A$ is
left $\!l_S\!$-strong and right $\!l_T\!$-strong,
while useful in finding elements of $A$ whose images
lie in simple submodules of $C$ {\em outside} of given
submodules, and elements whose kernels have the dual property,
does not provide a way of finding such elements whose images lie
together {\em in} some proper submodule (and dually for kernels).
But that was what we needed in Case~1 of the proof
of Proposition~\ref{P.cm_bilin}, when we
chose $b_1,\dots,b_{m-1}\in \r{ker}(f).$

However, I think the concept of an $\!N\!$-strong map can be modified
to make it compatible with
the method of proof of Proposition~\ref{P.cm_bilin}.
Let me sketch how.

Given $W,$ $X$ and $Y$ as in Definition~\ref{D.lt_strong},
and adding the assumption that $Y$ has finite length,
our modified definition will involve a concept of $W$ being left
$\!N\!$-strong {\em relative to} a nonzero submodule $Y'\subseteq Y.$
That condition will be defined by recursion on the length of $Y'.$
If $Y'$ has length $1,$ then $W$ will be called left $\!N\!$-strong
relative to $Y'$ if and only if $W$ has a nonzero element
whose image is contained in $Y'.$
If the concept has been defined
relative to submodules of some length $r\geq 1,$ then $W$ will be
said to be left $\!N\!$-strong relative to a submodule $Y'$
of length $r+1$ if and only if for every submodule $Y''\subseteq Y'$
of length $r-1,$ there exist $>N$ submodules of
length $r$ between $Y''$ and $Y'$ relative to which $W$ is
left $\!N\!$-strong.
We will simply say $W$ is left $\!N\!$-strong (in our new sense) if
it is left $\!N\!$-strong relative to its codomain $Y.$

It is now easy to verify that
if, for every simple submodule of $Y_0\subseteq Y,$ there is a
nonzero element of $W$ with image in $Y_0,$ and if the division ring
over which $V$ is a full matrix ring has cardinality $\geq N,$
then $W$ is left $\!N\!$-strong under our new definition.
Moreover, an easy induction shows that if
a union $W=W_1\cup\dots\cup W_d$ is left $\!N\!$-strong, then
at least one of the $W_i$ is left $\!N/d\!$-strong.

The definition of right $\!N\!$-strong would be modified analogously.

In the proof of Proposition~\ref{P.cm_bilin},
the hypotheses~\eqref{d.cm_b} and~\eqref{d.cm_c} could then be
weakened to say that (up to the change of notation appropriate
to a subspace $A\subseteq B\otimes_k C$ rather than a map
$A\times B\to C),$ $A$ is both left and right $\!1\!$-strong in our
modified sense.
I suspect that the same method could be adapted to the last
part of the proof of Proposition~\ref{P.main}, and would in fact
allow us to weaken the hypothesis~\eqref{d.cardD} by dropping
the factors $l_S$ and~$l_T.$

I haven't worked out the details, because they do not
get at the serious restrictions in our results.
In particular, the suggested change in the end of the proof
of Proposition~\ref{P.main}
would not get rid of the need for condition~\eqref{d.small_either}.

Let us now look at what we wish we could do.

\subsection{Some questions}\label{SS.question}
Here is an innocent-sounding generalization of our first main
result that we might ask for.

\begin{question}\label{Q.socle_central}
Does Theorem~\ref{T.TG} remain true if the assumption that $\soc(R)$
is central in $R$ is deleted?
\end{question}

To prove such a result, we would want a version of
Proposition~\ref{P.cm_bilin} involving a division ring
rather than a field.
As in the development of our results on non-local
rings, I don't see a convenient way of ``symmetrizing''
the general statement we need, so in the next question, I
will ask, not for the analog of that proposition, but of its corollary.
Moreover, although the result needed just to get a positive
answer to Question~\ref{Q.socle_central}
would have for $B$ and $C$ vector spaces over the
same division ring, I expect it would be no more difficult
to prove such a result without that restriction; so let us
pose the question as follows.

\begin{question}\label{Q.loc_bilin}
Let $S$ and $T$ be division rings,
${_S}B$ and ${_T}C$ be nonzero finite-dimensional vector spaces,
and ${_T}A_S$ be a subbimodule of the
$\!(T,S)\!$-bimodule of all additive group homomorphisms
${_S}B\to {_T}C.$
Suppose moreover that
\begin{equation}\begin{minipage}[c]{35pc}\label{d.loc_b}
For every proper subspace $B'\subsetneq B,$ there is at least one
$a\in A$ whose restriction to $B'$ is zero.
\end{minipage}\end{equation}
and
\begin{equation}\begin{minipage}[c]{35pc}\label{d.loc_c}
For every proper homomorphic image $C/C_0$ of $C,$ there is at least
one $a\in A$ whose composite with the factor map $C\to C/C_0$ is zero.
\end{minipage}\end{equation}
Then must it be true that
\begin{equation}\begin{minipage}[c]{35pc}\label{d.loc_lt}
$\lt({_T}A_S)\ \geq\ \dim_S(B)\ +\ \dim_T(C)\ -\ 1$
\end{minipage}\end{equation}
\textup{(}where $\lt({_T}A_S)$ denotes the length
of $A$ as a bimodule\textup{)}?
\end{question}

I have posed this question in a relatively easy-to-state form,
but my hope is that if a positive answer can be proved, then one
will be able push the proof further, and get the same result with the
hypotheses~\eqref{d.loc_b} and~\eqref{d.loc_c} weakened
to say that $A$ is left and right $\!N\!$-strong for
appropriate $N,$ in the modified sense
sketched in \S\ref{SS.better_N-str} above (or something
like it), and that this could be used to prove generalizations
of Proposition~\ref{P.main} and Theorem~\ref{T.main}.

Incidentally, it would be harmless to throw into
our hoped-for variant of
Proposition~\ref{P.main} the added assumption that $A$
is semisimple as a left module, a right module, and a bimodule,
since those conditions are true of the socle of
a left Artin ring, the situation to which we apply that
result in Theorem~\ref{T.main}.

\subsection{A waffle about wording}\label{SS.ubiquitous}
I would have preferred a more suggestive
term for what I have called left and right $\!N\!$-strong
families of maps.
I was tempted to replace ``strong'' with ``ubiquitous'';
but this might suggest too much -- a property like~\eqref{d.b}
and~\eqref{d.c}, rather than the weaker property actually defined.
Another thought was ``prevalent''; but this seemed a bit vague.

\section{Getting minimal faithful modules from module decompositions}\label{S.sub_or_hom}

This section is essentially independent of the rest of this note,
though the final assertion proved will complement
the results proved above.
Namely, Proposition~\ref{P.socle}(iii) below, though it has a
weaker conclusion than Theorems~\ref{T.TG} and~\ref{T.main},
is applicable when the hypotheses
of those theorems are not satisfied.
Aside from that final result, the focus will be on modules
with only one of the two minimality properties considered in
preceding sections.
(The one other exception to independence from the rest of this
note is that at one point below,
we will refer to an example from the preceding section.)

The two preliminary lemmas below
may be of interest in their own right.
The final statement of each describes how, under certain
conditions, a module can be decomposed into ``small pieces'':
in the first, as a sum of submodules $N$ such
that $N/J(R)N$ is simple; in the second, as a subdirect
product of modules $N$ with $\soc(N)$ simple.

\begin{lemma}\label{L.M_as_sum}
Let $R$ be a left Artinian ring, and $M$ a left $\!R\!$-module
\textup{(}not necessarily Artinian\textup{)}.
Then

\textup{(i)} If $L$ is a simple submodule of $M/J(R)M,$
then $M$ has a submodule $N$ such that the inclusion $N\subseteq M$
induces an isomorphism $N/J(R)N\cong L\subseteq M/J(R)M.$

Hence

\textup{(ii)} Given a decomposition of $M/J(R)M$
as a sum of simple modules $\sum_{i\in I} L_i,$ one
can write $M$ as the sum of a family of submodules
$N_i$ $(i\in I),$ such that for each $i,$ $N_i/J(R)N_i\cong L_i,$
and $L_i$ is the image of $N_i$ in $M/J(R)M.$
\end{lemma}

\begin{proof}
In the situation of~(i),
let $x$ be any element of $M$ whose image in $M/J(R)M$
is a nonzero member of $L.$
Thus, the image of $R\,x$ in $M/J(R)M$ is $L.$
Since $R$ is left Artinian, so is $R\,x,$ hence we can
find a submodule $N\subseteq R\,x$ minimal for having $L$ as its image
in $M/J(R)M.$
Now since $N/J(R)N$ is semisimple, it
has a submodule $L'$ which maps isomorphically to $L$ in $M/J(R)M.$
If $L'$ were a proper submodule of $N/J(R)N,$ then its inverse
image in $N$ would
be a proper submodule of $N$ which still mapped surjectively
to $L,$ contradicting the minimality of $N.$
Hence $N/J(R)N=L'\cong L,$ completing the proof of~(i).

In the situation of~(ii), choose for each $L_i$
a submodule $N_i\subseteq M$ as in~(i).
Then we see that $M=J(R)M+\sum_i N_i,$ hence since $R$ is Artinian,
$M=\sum_i N_i$ \cite[Theorem~23.16\,$(1)\!{\implies}\!(2')$]{TYL1},
as required.
\end{proof}

The next result is of a dual sort, but the arguments can be
carried out in a much more general context, so that the result we are
aiming for (the final sentence) looks like an afterthought.

\begin{lemma}\label{L.M_as_subdir}
Let $R$ be a ring and $M$ a left $\!R\!$-module.
Then

\textup{(i)} If $L$ is any submodule of $M,$ then $M$
has a homomorphic image $M/N$ such that the composite
map $L\hookrightarrow M\to M/N$ is an embedding, and
the embedded image of $L$ is essential in $M/N$ \textup{(}i.e.,
has nonzero intersection with every nonzero submodule of $M/N).$

Hence

\textup{(ii)} If $E$ is an essential submodule of $M,$
and $f: E\to\prod_I L_i$ a subdirect decomposition of $E,$
then there exists a subdirect decomposition $g:M\to\prod_i M_i$
of $M,$ such that each $M_i$ is an overmodule of $L_i$ in which $L_i$
is essential, and $f$ is the restriction of $g$ to $E\subseteq M.$

In particular, every locally Artinian module can be written
as a subdirect product of locally Artinian modules with simple socles.
\end{lemma}

\begin{proof}
In the situation of~(i), let $N$ be maximal among submodules
of $M$ having trivial intersection with $L.$
The triviality of this intersection means that $L$ embeds
in $M/N,$ while the maximality condition makes
the image of $L$ essential therein.
(If it were not essential, $M/N$ would have a
nonzero submodule $M'$ disjoint from the image of $L,$ and the
inverse image of $M'$ in $M$ would contradict the maximality of $N.)$

In the situation of~(ii), for each $j\in I$ let $K_j$ be the kernel of
the composite $E\to\prod_I L_i\to L_j.$
Applying statement~(i) with $M/K_j$ in the role of $M,$ and
$E/K_j\cong L_j$ in the role of $L,$ we get an
image $M_j$ of $M/K_j,$ and
hence of $M,$ in which $L_j$ is embedded and is essential.
Now since $E$ is essential in $M,$ every
nonzero submodule $M'\subseteq M$ has
nonzero intersection with $E,$ and that
intersection has nonzero projection to $L_i$ for some $i;$
so in particular, for that $i,$ $M'$ has nonzero image in $M_i.$
Since this is true for every $M',$
the map $M\to\prod_I M_i$ is one-to-one,
and gives the desired subdirect decomposition.

To get the final assertion, note that the socle
of a locally Artinian module is essential, and, being semisimple,
can be written as a subdirect product (indeed, as a direct sum)
of simple modules; so we can apply~(ii) with $E=\soc(M)$
and the $L_i$ simple.
Each of the $M_i$ in the resulting decomposition will
have a simple essential submodule $L_i,$ so that submodule
must be its socle.
\end{proof}

We can now get the following result, showing that given a
faithful module $M$ over an Artinian ring, we can carve
out of $M$ a ``small'' faithful submodule,
factor-module, or subfactor.
Note that in the statement, the length of $\soc(R)$ as a bimodule
may be less than its length as a left or right module.
(For example, the full $n\times n$ matrix
ring over a division ring is its own socle, and has length
$n$ as a right and as a left module, but length $1$ as a bimodule.
Similarly, in the $R$ of \S\ref{SS.small},
each of the direct summands $x\,F$ and $F\,y$ of $\soc(R)$
has length $2$ as left and as right $\!R\!$-module, but length $1$
as an $\!R\!$-bimodule; so $\soc(R)$ has length $4$
on each side, but $2$ as a bimodule.)

\begin{proposition}\label{P.socle}
Let $R$ be a left Artinian ring,
let $n$ be the length of $\soc(R)$ as a bimodule
\textup{(}equivalently, as a $\!2\!$-sided ideal\textup{)},
and let $M$ be a faithful $\!R\!$-module.
Then

\textup{(i)} $M$ has a submodule $M'$ which is again faithful over $R,$
and satisfies $\lt(M'/J(R)\,M')\leq n.$
\textup{(}In particular, $M'$ is generated by $\leq n$
elements.\textup{)}

\textup{(ii)} $M$ has a homomorphic image $M''$ which is faithful
over $R,$ and satisfies $\lt(\soc(M''))\leq\nolinebreak n.$

\textup{(iii)} $M$ has a subfactor faithful
over $R$ satisfying both these inequalities.
\end{proposition}

\begin{proof}

To get~(i), note that since $R/J(R)$ is semisimple Artin,
$M/J(R)M$ can be written as a direct sum of simple
modules $L_i$ $(i\in I)$ over that ring, and hence over $R,$
so we can construct a generating family of submodules
$N_i\subseteq M$ related to these as in Lemma~\ref{L.M_as_sum}(ii).
Since $M=\sum_i N_i$ is faithful, and $\soc(R)$ has length $n$ as a
$\!2\!$-sided ideal, the sum of some family of $\leq n$
of these submodules, say
\begin{equation}\begin{minipage}[c]{35pc}\label{d.M'}
$M'\ =\ N_{i_1}+\dots+N_{i_m}$\quad where $m\leq n,$
\end{minipage}\end{equation}
must have the property that $M'$ is annihilated by no
nonzero subideal of $\soc(R).$
(Details: one chooses the $N_{i_j}$ recursively;
as long as $N_{i_1}+\dots+N_{i_j}$ is annihilated by a nonzero
subideal $I\subseteq\soc(R),$ one can choose an $N_{i_{j+1}}$
which fails to be annihilated by $I.$
The annihilators in $\soc(R)$ of successive
sums $M'=N_{i_1}+\dots+N_{i_j}$ $(j=0,1,\dots)$
form a strictly decreasing chain, so this chain
must terminate after $\leq n$ steps.)
Since an ideal of an Artinian ring having zero intersection
with the socle is zero, $M'$ has zero annihilator, i.e., is faithful.
Since each $N_i$ satisfies $\lt(N_i/J(R)N_i)=1,$
we have $\lt(M'/J(R)M')\leq m\leq n.$

Statement (ii) is proved in the analogous way
from the final statement of Lemma~\ref{L.M_as_subdir}, using
images of $M$ in products of finite subfamilies
of the $M_i$ in place of submodules of $M$
generated by finite subfamilies of the $N_i.$

Statement (iii) follows by combining~(i) and~(ii).
\end{proof}

\section{Acknowledgements}\label{S.ackn}
I am indebted to Luchezar Avramov
for pointing me to Gulliksen's \cite[Lemma~2]{TG},
to Cl\'{e}ment de Seguins Pazzis for providing,
at \cite{overflow}, the proof of the present version
of~Proposition~\ref{P.cm_bilin}, and to Pace Nielsen
for many helpful comments on an earlier draft of this note.

\end{document}